\documentclass[letterpaper, 10 pt, conference]{ieeeconf}  
\usepackage{lineno}
\usepackage[cmex10]{amsmath}
\usepackage{amscd,amsfonts,amsmath,amssymb,mathrsfs,pifont,stmaryrd,tipa}
\usepackage{array,cases,dsfont,graphicx,texdraw}
\usepackage{wrapfig}
\usepackage{graphics} 
\usepackage{epsfig} 
\usepackage{stfloats}
\usepackage{subfig}
\usepackage{hyperref}
\usepackage{color}

\newtheorem{Remark}{Remark}
\newtheorem{Corollary}{Corollary}

\newtheorem{Theorem}{Theorem}

\newtheorem{Assumption}{Assumption}

\input{tcilatex}                                                          

\IEEEoverridecommandlockouts                              
\title{\LARGE \bf
Distributed Robust Seeking of Nash Equilibrium for Networked Games: An Extended State Observer based Approach}

\author{Maojiao Ye
\thanks{M. Ye is with the School of Automation, Nanjing University of Science and Technology, Nanjing 210094, P.R. China (Email: ye0003ao@e.ntu.edu.sg).}
\thanks{This work is supported by the National Natural Science Foundation of China (NSFC), No. 61803202, and the Natural Science Foundation of
Jiangsu Province, No. BK20180455.}
\thanks{Part of the manuscript was presented at IEEE Conference on Decision and Control, 2019 \cite{YECDC19}.}
}
\begin{document}

\maketitle
\thispagestyle{empty}
\pagestyle{empty}

\begin{abstract}
This paper aims to accommodate networked games in which the players' dynamics are subject to unmodeled and disturbance terms. The unmodeled and disturbance terms are regarded as extended states for which observers are designed to estimate them. Compensating the players' dynamics with the observed values, the control laws are designed to achieve the robust seeking of the Nash equilibrium for networked games. Firstly, we consider the case in which the players' dynamics are subject to time-varying disturbances only. In this case, the seeking strategy is developed by employing a smooth observer based on the Proportional-Integral (PI) control. By utilizing the designed strategy, we show that the players' actions would converge to a small neighborhood of the Nash equilibrium. Moreover, the ultimate bound can be adjusted to be arbitrarily small by tuning the control gains. Then, we further consider the case in which both an unmodeled term and a disturbance term coexist in the players' dynamics. In this case, we adapt the idea from the Robust Integral of the Sign of the Error (RISE) method in the strategy design to achieve the asymptotic seeking of the Nash equilibrium. Both strategies are analytically investigated via the Lyapunov stability analysis. The applications of the proposed methods for a network of velocity-actuated vehicles are discussed. Finally, the effectiveness of the proposed methods is verified via conducting numerical simulations.
\end{abstract}


\section{Introduction}
More and more researchers have begun to explore Nash equilibrium computation for networked games, especially in recent years (see e.g., \cite{YeTAC18}-\cite{YeTcyber18} for more comprehensive reviews). The main motivations for their efforts lie in both the wide applications of noncooperative games and their enormous potential to be adapted for the coordinative control and optimization of networked systems. The main focuses of the existing works are on algorithm developments, communication issues and tackling the constraints of the games in which the players are of first-order integrator-type dynamics. For example, Nash equilibrium seeking strategies were proposed for general networked games, $N$-cluster noncooperative games and (semi-)aggregative games in \cite{YeTAC18}\cite{SalehisadaghianiAT16}, \cite{YeAT18}-\cite{YEAT2020} and \cite{YECyber2017}-\cite{Belgioioso17}, respectively.
The authors in \cite{ZhuAUTO16} considered Nash equilibrium computation for generalized convex games. In particular, the robustness of the Nash equilibrium computing method towards the changes of the network topologies as well as the transmission delays was investigated. In \cite{YeTcyber18}, the robustness of the consensus-based Nash computation algorithm towards switching communication topologies among the players was explored. In addition, robust analysis for the loss of communication was conducted. Given that the lengths of the time slots, in which the communication among the players was lost, were bounded by some certain value, the method can still achieve Nash seeking under certain conditions \cite{YeTcyber18}. \textbf{Nevertheless, it is worth mentioning that in many existing works including \cite{YeTAC18}-\cite{YeTcyber18}, the players were considered to be static (or first-order integrator-type) players whose action updating strategies can be freely designed though the system dynamics exhibit much more complex characteristics in many engineering systems.}

In practical situations, the dynamics of the plants are usually subject to unmodeled terms and disturbances. For example, due to external disturbance and imprecise microfabrication, the control strategy designed for the MEMS gyroscope should take the parameter variations, the mechanical-thermal noises and the axes mechanical couplings into account \cite{ZhengTCST09}. To compensate for the effects of imprecise microfabrication and external disturbance, an Active Disturbance Rejection Control (ADRC) based strategy was adapted for the control of MEMS gyroscope in \cite{ZhengTCST09}. Moreover, physical hydraulic systems were trapped by uncertain parameters (e.g., hydraulic and payload unknown parameters) and external disturbances \cite{YaoTASCE17}. Gear shifting and tank pressure resulted in disturbance and varying parameters in the regenerative braking torque control of the air hybrid vehicles \cite{Fazeli12}. Moreover, it is a challenging task to get the accurate model information for the pneumatic muscle actuators. To address the unmodeled complexities, an ADRC-based control strategy was proposed for the pneumatic muscle actuator systems in \cite{YuanTIE19}. Hydrodynamic damping forces and the working environments, that the autonomous marine surface vessels suffer, gave rise to unmodeled dynamics and unpredictable disturbances for the tracking control of marine surface vessels \cite{HeTIE19}. \textbf{Motivated by the fact that unmodeled dynamics and disturbances are practically ubiquitous, this paper considers the robust Nash equilibrium computation for games in which there are unmodeled dynamics and time-varying disturbances in the players' dynamics.}

To deal with uncertainties and disturbances in system dynamics, various methods (e.g, disturbance-observer-based control, equivalent input disturbance estimator and extended state observers, to mention just a few) have been proposed \cite{Chen16}. For instance, in \cite{Chen04}, disturbance-observer-based control techniques were utilized for nonlinear systems subject to
disturbances that are produced by an exogenous system. In \cite{SheTIE08}, an equivalent input disturbance estimator was proposed for a servo system. A generalized extended state observer based approach was investigated for systems that are suffering from mismatched uncertainties in \cite{Li12}. The core idea of these disturbance attenuation methods is to estimate the uncertainties and disturbances and then compensate the system dynamics with the evaluated value \cite{Chen16}. Noticing that the extended state observer based approaches require least plant information among the aforementioned disturbance estimation and attenuation methods \cite{Li12}, this paper aims to design distributed robust Nash equilibrium seeking strategies for general networked games investigated in \cite{YeTAC18}\cite{YeTcyber18} via designing extended state observer based approaches.  Compared with the distributed seeking algorithms in \cite{YeTAC18}\cite{YeTcyber18}, this paper accommodates the unmodeled complexities and time-varying disturbances in the system dynamics. The unmodeled dynamics bring couplings among all the players' dynamics, thus making the problem much more challenging.
 In summary, with part of the manuscript presented in \cite{YECDC19}, the main contributions of the paper are given in threefold:
\begin{enumerate}
  \item The robust Nash equilibrium seeking problem is considered for games under distributed communication networks. In the considered problem, there exist unmodeled complexities and time-varying disturbances in the players' dynamics.  In the proposed methods, the uncertainties and disturbances in the system dynamics are regarded as extended states, for which a PI-based observer and a RISE-based observer are adapted. Based on the designed observers, distributed robust Nash equilibrium seeking strategies are established.
  \item The convergence results are theoretically established via conducting Lyapunov stability analysis. It is shown that driven by the PI-based algorithm, the players' actions would converge to a neighborhood of the Nash equilibrium. Moreover, the ultimate bound can be adjusted to be arbitrarily small by tuning the control gains. For the RISE-based method, it is theoretically proven that the players' actions can be steered to the Nash equilibrium asymptotically.
  \item The applications of the proposed methods for a network of velocity-actuated vehicles are discussed. In particular, the convergence results for the connectivity control game among networked mobile sensors are established.
\end{enumerate}

We proceed the remainder of the paper in the following order. The preliminaries and notations are given in Section \ref{pre}. The problem is formulated in Section \ref{p_f}. The main results are given in Section \ref{p1_res}, where a PI-based method and a RISE-based method are developed. The applications of the proposed methods in sensor networks are discussed in Section \ref{v_5}. A numerical example is given in Section \ref{p_1_numer} for the verification of the proposed methods and the conclusions are drawn in Section \ref{conc}.

\section{Preliminaries and notations}\label{pre}
In this paper, we use $\mathcal{G}=(\mathcal{V},\mathcal{E}_d)$ to denote a communication graph.  Note that $\mathcal{V}=\{1,2,\cdots,M\}$, where $M\geq 2$ and is an integer, defines the set of nodes in the graph and $\mathcal{E}_d$ defines the set of edges for the nodes. The elements of $\mathcal{E}_d$ are $(i,j)$, which represents for an edge from node $i$ to node $j.$ Associate with $(i,j)\in \mathcal{E}_d$ a weight $a_{ji}>0$ and $a_{ii}=0$. If $a_{ij}=a_{ji},\forall i,j\in \mathcal{V},$ we say that $\mathcal{G}$ is undirected. A matrix with its element on the $i$th row and $j$th column being $a_{ij}$ ($a_{ij}>0$ if $(j,i)\in \mathcal{E}_d,$ else, $a_{ij}=0$) is the adjacency matrix of graph $\mathcal{G}$ and is termed as $\mathcal{A}$. Agent $j$ is a neighbor of agent $i$ if $a_{ij}>0$. Given that there is a path between any pair of distinct vertices, the undirected graph is connected. Define $\mathcal{D}$ as a diagonal matrix whose $i$th diagonal element is $\sum_{j=1}^{M}a_{ij}$. Then, the Laplacian matrix $\mathcal{L}$ is equal to $\mathcal{D}-\mathcal{A}$ \cite{YeTAC18}\cite{YETCST16}.

Notations: The concatenated column vector form of $h_i$ for $i\in\{1,2,\cdots,N\}$ is denoted as $[h_i]_{vec}$. Similarly, $[h_{ij}]_{vec}$ ($\text{diag}\{h_{ij}\}$) for $i\in\{1,2,\cdots,N\},j\in\{1,2,\cdots,M\}$ is defined as a column vector (diagonal matrix) whose elements (diagonal elements) are $h_{11},h_{12},\cdots,h_{1M},h_{21},\cdots,h_{NM},$ successively. The minimum and maximum eigenvalues of a symmetric positive definite matrix $P$ are denoted as $\lambda_{min}(P)$ and $\lambda_{max}(P),$ respectively. Moreover, the minimum and maximum elements of $h_i$ for $i\in\{1,2,\cdots,N\}$ are denoted as $\min\{h_i\}$ and $\max\{h_i\},$ respectively. The notation $\min\{a,b\}=a$ $(\max\{a,b\}=b)$ if $a\leq b.$ Else, $\min\{a,b\}=b$ $(\max\{a,b\}=a).$

\section{Problem Formulation}\label{p_f}
In this paper, we consider a game with $N$ players labeled from $1$ to $N$, successively. In the considered game, player $i$ intends to
\begin{equation}\label{from}
\begin{aligned}
&\text{min}_{x_i} \ \ f_i(\mathbf{x})\\
&\text{subject to} \ \ \dot{x}_i=u_i+\varsigma g_i(\mathbf{x})+d_i(t),
\end{aligned}
\end{equation}
where $i\in\mathcal{N}$ and $\mathcal{N}=\{1,2,\cdots,N\}$ is the player set. Moreover, $x_i\in \mathbb{R}$, $u_i$ and $f_i(\mathbf{x})$ are the action, the control input and the objective function of player $i$, respectively. In addition, $\mathbf{x}=[x_1,x_2,\cdots,x_N]^T$. Note that $g_i(\mathbf{x})$ and $d_i(t)$ respectively denote the unmodeled term and the external disturbance whose explicit expressions are not available. The parameter $\varsigma$ is equal to $0$ or $1$. In the rest of the paper, these two cases will be investigated successively. Suppose that if player $j$ is not a neighbor of player $i,$ then player $i$ can not directly access the action of player $j$.  The paper aims to design the control inputs to achieve robust Nash equilibrium seeking for the considered game.

For notational convenience, define $\mathbf{x}_{-i}=[x_1,x_2,\cdots,x_{i-1},x_{i+1},\cdots,x_N]^T$. Then, $f_i(\mathbf{x})$ can be written as $f_i(x_i,\mathbf{x}_{-i})$ instead. Following the above notations, Nash equilibrium is defined as follows.
\begin{definition}\label{def_1}
An action profile $\mathbf{x}^*=(x_i^*,\mathbf{x}_{-i}^*)$ is a Nash equilibrium if for $i\in\mathcal{N},$
\begin{equation}
f_i(x_i^*,\mathbf{x}_{-i}^*)\leq f_i(x_i,\mathbf{x}_{-i}^*),
\end{equation}
for $x_i\in\mathbb{R}$ \cite{YeTAC18}.
\end{definition}

\begin{Remark}
In this paper,
\begin{equation}\label{eq15}
\dot{x}_i=u_i+\varsigma g_i(\mathbf{x})+d_i(t),
\end{equation}
where $\varsigma=0$ or $\varsigma=1$. In contrast,
\begin{equation}
\dot{x}_i=u_i,
\end{equation}
in \cite{YeTAC18}\cite{YeTcyber18}, indicating that compared with \cite{YeTAC18}\cite{YeTcyber18}, unmodeled dynamics and uncertainties are further addressed in the paper. Note that as $g_i(\mathbf{x})$ is a function of $\mathbf{x},$ it brings couplings among the players' dynamics thus making the problem more challenging.
Moreover, the problem considered in this paper differs from that of \cite{YeAT18}-\cite{YEAT2020} in the following two aspects:
\begin{enumerate}
  \item  \textbf{On the game model}: The works in \cite{YeAT18}-\cite{YEAT2020} considered $N$-cluster games in which the engaged agents are separately contained in $N$ clusters. Each agent has a local objective function, which is a function of all the engaged agents' actions in the $N$ clusters. Moreover, the agents in the same cluster try to minimize the sum of their local objective functions. Different from \cite{YeAT18}-\cite{YEAT2020},  this paper considers that each player intends to minimize its own objective function, which is a function of all the players' actions, by adjusting its own action as described in \eqref{from}.
  \item \textbf{On the players' dynamics}: This paper considers that the players' dynamics are suffering from disturbances and unmodeled dynamics while  in \cite{YeAT18}-\cite{YEAT2020}, no dynamical constraints were considered and the Nash equilibrium seeking strategies can be freely designed.
\end{enumerate}

\end{Remark}

Adopted from \cite{YeTAC18}\cite{YeAT18}, the assumptions given below will be utilized in the upcoming sections.

\begin{Assumption}\label{ASS1}
For each $i\in\mathcal{N},$ $f_i(\mathbf{x})$ is a twice-continuously differentiable function.
\end{Assumption}
\begin{Assumption}\label{Assu2}
There exists a positive constant $m$ such that for $\mathbf{x},\mathbf{z}\in \mathbb{R}^N,$
\begin{equation}
(\mathbf{x}-\mathbf{z})^T(\Upsilon(\mathbf{x})-\Upsilon(\mathbf{z}))\geq m||\mathbf{x}-\mathbf{z}||^2,
\end{equation}
where $\Upsilon(\mathbf{x})=\left[\frac{\partial f_i(\mathbf{x})}{\partial x_i}\right]_{vec}.$
\end{Assumption}

\begin{Remark}
If $f_1(\mathbf{x})=f_2(\mathbf{x})=\cdots=f_N(\mathbf{x})=f(\mathbf{x}),$ Assumption \ref{Assu2} would be reduced to the strong convexity of $f(\mathbf{x}),$ which suggests the existence and uniqueness of the minimization solution of $f(\mathbf{x}).$ Similarly, Assumption \ref{Assu2} indicates that the pseudo-gradient vector $\Upsilon(\mathbf{x})$ is strongly monotone with constant $m,$ which suggests the existence and uniqueness of the Nash equilibrium in the game \cite{YeAT18}. Moreover, by Assumption \ref{Assu2}, the players' actions can be driven to the Nash equilibrium via the gradient play \cite{YeTAC18}, which lays the foundation for the stability analysis in this paper. Note that Assumption \ref{Assu2} is a widely adopted assumption in existing works on Nash equilibrium seeking (see e.g., \cite{YeTAC18}\cite{YeAT18} and the references therein).
\end{Remark}

\begin{Assumption}\label{ass_3}
The players are equipped with an undirected and connected communication graph $\mathcal{G}$.
\end{Assumption}

\section{Main Results}\label{p1_res}
In this section,  we firstly consider the case in which the players' dynamics are subject to external disturbance only, followed by the case in which the players' dynamics are subject to both an unmodeled term and a disturbance term.
\subsection{Strategy design with a smooth state observer}\label{part_1}
In this section, we concern with games in which player $i$'s action is given by
\begin{equation}\label{dyna_g}
\dot{x}_i=u_i+d_i(t),
\end{equation}
where $i\in\mathcal{N}.$ The disturbance $d_i(t)$ is supposed to satisfy the following condition.
\begin{Assumption}\label{ass_4}
The disturbance term $d_i(t)$ is continuously differentiable and $\dot{d}_i(t)$ is bounded for $t\geq t_0,i\in\mathcal{N},$ where $t_0$ is the initial time instant.
\end{Assumption}

To further proceed to the strategy design, let $z_i=d_i(t)$. Then, the dynamics of player $i$ can be written as
\begin{equation}
\begin{aligned}
\dot{x}_i=&u_i+z_i\\
\dot{z}_i=&\dot{d}_i(t).
\end{aligned}
\end{equation}

By regarding the external disturbance as an extended state, the control law is designed as
\begin{equation}\label{eq8}
u_i=-\frac{\partial f_i}{\partial x_i}(\mathbf{y}_i)-\hat{z}_i,
\end{equation}
in which $\frac{\partial f_i}{\partial x_i}(\mathbf{y}_i)=\frac{\partial f_i(\mathbf{x})}{\partial x_i}|_{\mathbf{x}=\mathbf{y}_i},$ $\mathbf{y}_i=[y_{i1},y_{i2},\cdots,y_{iN}]^T$ and $y_{ij},\hat{z}_i$ are generated by
\begin{equation}\label{eq9}
\begin{aligned}
\dot{y}_{ij}=&-\theta_{ij}(\sum_{k=1}^Na_{ik}(y_{ij}-y_{kj})+a_{ij}(y_{ij}-x_j))\\
\dot{\hat{x}}_i=&u_i+\hat{z}_i+\bar{k}_{i1}(x_i-\hat{x}_i)\\
\dot{\hat{z}}_i=&\bar{k}_{i2}(x_i-\hat{x}_i).
\end{aligned}
\end{equation}
Moreover, $\bar{k}_{i1}=\sigma k_{i1},\bar{k}_{i2}=\sigma^2 k_{i2},\theta_{ij}=\theta \bar{\theta}_{ij}$ where $\sigma, \theta$ are positive parameters to be determined and $\bar{\theta}_{ij},k_{i1},k_{i2}$ are fixed positive constants.

\begin{Remark}
As the players' dynamics are subject to time-varying disturbances, the main idea is to employ an observer to estimate the disturbance and then compensate for the disturbance by utilizing the estimated value in the control law. Hence, $\hat{z}_i$ can be regarded as the observed disturbance. Moreover, as the players do not have direct access into the actions of the players who are not their neighbors, this paper follows our previous works in \cite{YeTAC18}\cite{YeAT18} to estimate the requested information by utilizing consensus protocols.
\end{Remark}

Define
\begin{equation}
\begin{aligned}
\zeta_{i1}&=x_i-\hat{x}_i, \zeta_{i2}=z_i-\hat{z}_i\\
\xi_{i}&=x_i-x_i^*,\eta_{ij}=y_{ij}-x_j,
\end{aligned}
\end{equation}
where $x_i^*$ is defined in Definition \ref{def_1},
then,
\begin{equation}
\begin{aligned}
\dot{\zeta}_{i1}&=\dot{x}_i-\dot{\hat{x}}_i=-\bar{k}_{i1}\zeta_{i1}+\zeta_{i2}\\
\dot{\zeta}_{i2}&=\dot{z}_i-\dot{\hat{z}}_i=-\bar{k}_{i2}\zeta_{i1}+\dot{d}_i(t)\\
\dot{\xi}_{i}&=\dot{x}_i=-\frac{\partial f_i}{\partial x_i}(\mathbf{y}_i)+\zeta_{i2}\\
\dot{\eta}_{ij}&=\dot{y}_{ij}-\dot{x}_j\\
&=-\theta_{ij}(\sum_{k=1}^Na_{ik}(\eta_{ij}-\eta_{kj})+a_{ij}\eta_{ij})-\dot{\xi}_j.
\end{aligned}
\end{equation}

Hence,
\begin{equation}\label{clo_sys}
\begin{aligned}
\dot{\mathbf{\zeta}}_1&=-\sigma \mathbf{k}_1\mathbf{\zeta}_1+\mathbf{\zeta}_2\\
\dot{\mathbf{\zeta}}_2&=-\sigma^2 \mathbf{k}_2 \mathbf{\zeta}_1+\dot{\mathbf{d}}(t)\\
\dot{\mathbf{\xi}}&=-\left[\frac{\partial f_i}{\partial x_i}(\mathbf{y}_i)\right]_{vec}+\mathbf{\zeta}_2\\
\dot{\mathbf{\eta}}&=-\theta \bar{\Theta}(\mathcal{L}\otimes I_{N\times N}+\mathcal{A}_0)\mathbf{\eta}-\mathbf{1}_N\otimes \dot{\mathbf{\xi}},
\end{aligned}
\end{equation}
where $\mathbf{k}_1=\text{diag}\{k_{i1}\},\mathbf{k}_2=\text{diag}\{k_{i2}\},\mathbf{d}(t)=[d_i(t)]_{vec},\mathbf{\zeta}_1=[\zeta_{i1}]_{vec},\mathbf{\zeta}_2=[\zeta_{i2}]_{vec},\mathbf{\eta}=[\eta_{ij}]_{vec},\mathbf{\xi}=[\xi_{i}]_{vec},$ $\bar{\Theta}=\text{diag}\{\bar{\theta}_{ij}\}$, $\mathcal{A}_0=\text{diag}\{a_{ij}\}$ and $I_{N\times N}$ is an identity matrix of dimension $N\times N.$

Define $\mathbf{\zeta}=[\mathbf{\zeta}_1^T,\mathbf{\zeta}_2^T]^T$, then
\begin{equation}\label{eq6}
\dot{\mathbf{\zeta}}=-\Delta \mathbf{\zeta}+\mathbf{d}_1(t),
\end{equation}
where $\Delta=\left[
                \begin{array}{cc}
                  \sigma \mathbf{k}_1 & -I_{N\times N} \\
                  \sigma^2 \mathbf{k}_2 & \mathbf{0}_{N\times N} \\
                \end{array}
              \right]
$ and $\mathbf{d}_1(t)=[\mathbf{0}_N^T, \dot{\mathbf{d}}(t)^T]^T.$


It is obvious that the closed-loop system in \eqref{clo_sys} can be regarded as a cascaded system as $\dot{\mathbf{d}}(t)$ is bounded by  Assumption \ref{ass_4}. Hence, it is expectable that the proposed method can drive the players' actions to arbitrarily small neighborhood of the Nash equilibrium if the PI-based observer can achieve the disturbance estimation with arbitrarily small estimation error. In the following, we present the analytical convergence result for the closed-loop system in \eqref{clo_sys}.

\begin{Theorem}\label{th1}
Suppose that Assumptions \ref{ASS1}-\ref{ass_4} are satisfied. Then, for any positive constants $v_1,v_2$, there exists a positive constant $\theta^*(v_1,v_2),\sigma^*(v_1,v_2)$ such that for each $\theta\in(\theta^*,\infty)$, $\sigma\in(\sigma^*,\infty)$, there exists $\bar{T}\geq t_0$ such that for $t\geq \bar{T},$
\begin{equation}
||\mathcal{E}(t)||\leq v_2,
\end{equation}
given that $||\mathcal{E}(t_0)||\leq v_1,$ where $\mathcal{E}=[\mathbf{\zeta}^T,\mathbf{\eta}^T,\mathbf{\xi}^T]^T$.
\end{Theorem}
\begin{proof}
See Section \ref{prof_th1} for the proof.
\end{proof}

In the analysis of Theorem \ref{th1}, we suppose that $f_i(\mathbf{x})$ is locally Lipshcitz for $i\in\mathcal{N}$. Actually, given that $f_i(\mathbf{x})$ is globally Lipshcitz for $i\in\mathcal{N}$, Theorem \ref{th1} can be further strengthened. In the following corollary,  we provide the corresponding result for the case in which $f_i(\mathbf{x})$ is globally Lipshcitz for $i\in\mathcal{N}$.
\begin{Corollary}\label{co1}
Suppose that  Assumptions \ref{ASS1}-\ref{ass_4} are satisfied and $\frac{\partial f_i(\mathbf{x})}{\partial x_i}$ for all $i\in\mathcal{N}$ are globally Lipschitz. Then, for any positive constant $v_1$, there exist positive constants $\theta^*,\sigma^*(v_1)$ such that for each $\theta\in(\theta^*,\infty)$, $\sigma\in(\sigma^*,\infty)$, there exists $\bar{T}\geq t_0$ such that for $t\geq \bar{T},$
\begin{equation}
||\mathcal{E}(t)||\leq v_1.
\end{equation}
\end{Corollary}
\begin{proof}
The result can be derived by noticing that if $\frac{\partial f_i(\mathbf{x})}{\partial x_i}$ for all $i\in\mathcal{N}$ are globally Lipschitz, the proof of Theorem \ref{th1} holds for any initial condition.
\end{proof}

\begin{Remark}
In Theorem \ref{th1} and Corollary \ref{co1}, we establish the convergence results for the proposed method in \eqref{eq8}-\eqref{eq9} and show that the proposed method can drive the players' actions to an arbitrarily small neighborhood of the Nash equilibrium by tuning the control parameters. Comparing Theorem \ref{th1} with Corollary \ref{co1}, it can be seen that if $\frac{\partial f_i(\mathbf{x})}{\partial x_i}$ for all $i\in\mathcal{N}$ are globally Lipschitz, we can obtain a global convergence result. However, if $\frac{\partial f_i(\mathbf{x})}{\partial x_i}$ is locally Lispchitz for $i\in\mathcal{N}$, it is required that the initial errors should be bounded (though the bound can be arbitrarily large).
\end{Remark}

In the following section, an asymptotic seeker that can accommodate both the unmodeled complexities and time-varying disturbances in the system dynamics will be presented.
\subsection{Strategy design with an asymptotic state observer}
In this section, we consider the Nash equilibrium seeking in which
\begin{equation}
\dot{x}_i=u_i+ g_i(\mathbf{x})+d_i(t).
\end{equation}

The upcoming analysis will be proceeded based on the following assumption.
\begin{Assumption}\label{ass_5}
The functions $g_i(\mathbf{x})$ for $i\in\mathcal{N}$ are sufficiently smooth. Moreover, $\frac{\partial g_i(\mathbf{x})}{\partial x_j}$ and $\frac{\partial^2 g_i(\mathbf{x})}{\partial x_j\partial x_k}$ for $i,j,k \in\mathcal{N}$ are bounded given that $\mathbf{x}$ is bounded. In addition, the disturbances $d_i(t)$ for $i\in\mathcal{N}$ are sufficiently smooth and $\dot{d}_i(t),\ddot{d}_i(t)$ are bounded for $i\in\mathcal{N},t\geq t_0.$
\end{Assumption}

To accommodate the unmodeled dynamics and disturbance term, let $z_i=g_i(\mathbf{x})+d_i(t).$ Then, the dynamics of player $i$ can be written as
\begin{equation}
\begin{aligned}
\dot{x}_i&=u_i+z_i\\
\dot{z}_i&=\dot{g}_i(\mathbf{x})+\dot{d}_i(t).
\end{aligned}
\end{equation}

To seek the Nash equilibrium, let
\begin{equation}\label{eq11}
u_i=-\frac{\partial f_i}{\partial x_i}(\mathbf{y}_i) -\hat{z}_i
\end{equation}
in which $\mathbf{y}_i=[y_{i1},y_{i2},\cdots,y_{iN}]^T$. Moreover,
\begin{equation}\label{eq12}
\begin{aligned}
\dot{y}_{ij}&=-\theta_{ij}(\sum_{k=1}^N a_{ik}(y_{ij}-y_{kj})+a_{ij}(y_{ij}-x_j))\\
\dot{\hat{x}}_i&=u_i+\hat{z}_i+(k_i^s+c_i)(x_i-\hat{x}_i)\\
\dot{\hat{z}}_i&=k_i^sc_i (x_i-\hat{x}_i)+\beta_i sgn(x_i-\hat{x}_i),
\end{aligned}
\end{equation}
where $\theta_{ij}=\theta \bar{\theta}_{ij}$, $\theta$ is a positive control gain to be further determined and $\bar{\theta}_{ij}$ is a fixed control parameter. Moreover, $sgn(\cdot)$ is the standard signum function and $k_i^s,c_i,\beta_i$ are positive control gains to be further determined.
\begin{Remark}
The proposed method is motivated by \cite{YeAT16} in which the unknown dynamics and disturbances are estimated  based on the RISE method (see, e.g., \cite{YeAT16}-\cite{XianTAC04}). Compared with the strategy in Section \ref{part_1}, we see that the main difference is that a signum function is further included in the observer to achieve better convergence results.
\end{Remark}

To facilitate the subsequent analysis, define the error signals as
\begin{equation}
\begin{aligned}
\zeta_{i1}&=x_i-\hat{x}_i, \zeta_{i2}=z_i-\hat{z}_i\\
\xi_i&=x_i-x_i^*, \eta_{ij}=y_{ij}-x_j.
\end{aligned}
\end{equation}

Then,
\begin{equation}
\begin{aligned}
\dot{\zeta}_{i1}=&\dot{x}_i-\dot{\hat{x}}_i\\
=&z_i-\hat{z}_i-(k_i^s+c_i)(x_i-\hat{x}_i)\\
=&-(k_i^s+c_i)\zeta_{i1}+\zeta_{i2}\\
\dot{\zeta}_{i2}=&\dot{z}_i-\dot{\hat{z}}_i\\
=&\dot{g}_i(\mathbf{x})+\dot{d}_i(t)-c_ik_i^s \zeta_{i1}-\beta_i sgn(\zeta_{i1}).
\end{aligned}
\end{equation}
Moreover,
\begin{equation}
\dot{\xi}_i=\dot{x}_i=-\frac{\partial f_i}{\partial x_i}(\mathbf{y}_i)+\zeta_{i2},
\end{equation}
and
\begin{equation}
\dot{\eta}_{ij}=-\theta_{ij}(\sum_{k=1}^Na_{ik}(\eta_{ij}-\eta_{kj})+a_{ij}\eta_{ij})-\dot{\xi}_j.
\end{equation}
Writing the error system in the concatenated vector form gives the observation error subsystem as
\begin{equation}\label{eq2}
\begin{aligned}
\dot{\mathbf{\zeta}}_1=&-(k_s+c)\mathbf{\zeta}_1+\mathbf{\zeta}_2\\
\dot{\mathbf{\zeta}}_2=&-ck_s \mathbf{\zeta}_1-\beta sgn(\mathbf{\zeta}_1)+[\dot{g}_i(\mathbf{x})]_{vec}+[\dot{d}_i(t)]_{vec},
\end{aligned}
\end{equation}
the optimization error subsystem as
\begin{equation}\label{eq5}
\dot{\mathbf{\xi}}=-\left[\frac{\partial f_i}{\partial x_i}(\mathbf{y}_i)\right]_{vec}+\mathbf{\zeta}_2,
\end{equation}
and the consensus error subsystem as
\begin{equation}\label{eq4}
\dot{\mathbf{\eta}}=-\theta \bar{\Theta}(\mathcal{L}\otimes I_{N\times N}+\mathcal{A}_0)\mathbf{\eta}-\mathbf{1}_N\otimes \dot{\mathbf{\xi}},
\end{equation}
where $k_s=\text{diag}\{k_i^s\}, c=\text{diag}\{c_i\},\beta=\text{diag}\{\beta_i\},$ $\mathbf{\zeta}_1=[\zeta_{i1}]_{vec},\mathbf{\zeta}_2=[\zeta_{i2}]_{vec},\mathbf{\xi}=[\xi_{ij}]_{vec},\mathbf{\eta}=[\eta_{ij}]_{vec},\bar{\Theta}=\text{diag}\{\bar{\theta}_{ij}\}$ and $\mathbf{y}=[y_{ij}]_{vec}$.

To facilitate the subsequent analysis, define a filtered signal as
\begin{equation}
\mathbf{\gamma}=\dot{\mathbf{\zeta}}_1+c \mathbf{\zeta}_1.
\end{equation}

Then,
\begin{equation}\label{eq3}
\begin{aligned}
\dot{\mathbf{\gamma}}=&-(k_s+c)\dot{\mathbf{\zeta}}_1+\dot{\mathbf{\zeta}}_2+c \dot{\mathbf{\zeta}}_1\\
=&-k_s\mathbf{\gamma}-\beta sgn(\mathbf{\zeta}_1)+[\dot{g}_i(\mathbf{x})]_{vec}+[\dot{d}_i(t)]_{vec}.
\end{aligned}
\end{equation}

Intuitively speaking, the signum function in the RISE-based method would enlarge the regulation force of the integration part for the observer when the error is very small thus leading to better control performance. Hence, we may expect enhanced stability  result for the RISE-based method compared with the result for the PI-based method in  \eqref{clo_sys}. Denote $\mathcal{E}_f=[\mathbf{\zeta}_1^T,\mathbf{\zeta}_2^T,\mathbf{\eta}^T,\mathbf{\gamma}^T]^T.$ In the following, we show that $||\mathcal{E}_f(t)||$ would be vanishing to zero as $t\rightarrow \infty$.
\begin{Theorem}\label{th2}
Suppose that Assumptions \ref{ASS1}-\ref{ass_3} and \ref{ass_5} are satisfied. Then, for any positive constant $v_1$, there exists a $\theta^*(v_1)>0$ such that for each $\theta\in(\theta^*,\infty)$, there exists a positive constant $k_s^r(v_1,\theta)$ such that for each $k_i^s\in(k_s^r,\infty),i\in\mathcal{N}$, there exists a positive constant $c^r(v_1,\theta,k_s)$ such that for each $c_i\in(c^r,\infty),i\in\mathcal{N},$ there exists a positive constant $\beta^r(v_1,\theta,c,k_s)$ such that for each $\beta_i\in(\beta^r,\infty),i\in\mathcal{N},$
\begin{equation}
||\mathcal{E}_f(t)||\rightarrow 0 \ \ \text{as} \ \ t\rightarrow \infty
\end{equation}
given that $||\mathcal{E}_f(t_0)||\leq v_1.$
\end{Theorem}
\begin{proof}
See Section \ref{prof_th2} for the proof.
\end{proof}
\begin{Remark}
Compared with the results in Section \ref{part_1}, we see that by utilizing the RISE-based method, asymptotic convergence results can be obtained though the system dynamics are subject to both uncertainties (i.e., $g_i(\mathbf{x})$) and disturbances (i.e., $d_i(t)$). Moreover, if the players' dynamics are governed by \eqref{dyna_g}, it can be seen that the RISE-based method would still result in an asymptotic stability result. Note that the PI-based method is given as it is smooth and simpler than the RISE-based method. In addition, compared with the PI-based method, stricter requirements on the disturbance are required for the RISE-based method (see Assumptions \ref{ass_4} and \ref{ass_5}). Furthermore, though it might be possible to apply the PI-based method to the case in which unmodeled dynamics and disturbance coexist in the players' dynamics by suitably tuning the control gains, the RISE-based method serves as a more precise control method compared with the PI-based method.
\end{Remark}

\begin{Remark}
In \cite{Yecyber16}, potential games with disturbance were solved by utilizing the
real-time measurements of the players' costs. Noticing that each player's objective
function might be determined by all the engaged players'
actions, this paper differs from \cite{Yecyber16} to utilize only local communication among the players, which benefits the applications in distributed systems. Different from \cite{Yecyber16} that only considered potential games with bounded disturbances, this paper accommodates general networked games with both disturbances and unmodeled dynamics.  Moreover, this paper proposes a method that can achieve robust asymptotic equilibrium seeking under distributed networks, while in \cite{Yecyber16}, we only showed that the players' actions can be driven to a neighborhood of the Nash equilibrium.
\end{Remark}

\begin{Remark}
With the increasing requirements on the precision of control algorithms, various disturbance estimation and attenuation methods have been proposed. For example, equivalent input disturbance based estimators, disturbance observer based control and extended state observers, to mention just a few, are typical effective methods for disturbance estimation and attenuation \cite{Chen16}. In this paper, the robustness of the Nash equilibrium seeking strategy is achieved by utilizing extended state observer based approaches as they require least plant information among the aforementioned disturbance estimation and attenuation methods \cite{Li12}. Moreover, the proposed extended state observers are designed based on PI or RISE, which are simple for practical implementation. If more plant information is available, it would be interesting future works to further reduce the order of the observers by utilizing other disturbance estimation and attenuation methods.
\end{Remark}

\section{Application to velocity-actuated mobile sensor networks}\label{v_5}
In \cite{Stankovic}, the authors defined a connectivity control game for a network of mobile sensors. By supposing that the sensors try to find a tradeoff between the local objective (e.g., source seeking, positioning) and the global objective (e.g., preserve connectivity with the other sensors), the cost function of sensor $i$ is defined as \cite{Stankovic}
\begin{equation}
J_i(\mathbf{x})=l^c_i(x_i)+\bar{l}^g_i(\mathbf{x}),
\end{equation}
where
\begin{equation}
l^c_i(x_i)=x_i^Tr_{ii}x_i+x_i^Tr_i+b_i,
\end{equation}
and
\begin{equation}\label{eq13}
\bar{l}^g_i(\mathbf{x})=\sum_{j\in\mathcal{N}_i}c_{ij}||x_i-x_j||^2.
\end{equation}
Moreover, $x_i=[x_{i1},x_{i2}]^T\in\mathbb{R}^2$ denotes the position of sensor $i$, $r_{ii},r_i,b_i,c_{ij}>0$ are constant matrices or vectors of compatible dimension and $\mathcal{N}_i$ denotes the neighbor set of sensor $i$ in the communication graph.  In addition, for each $i\in\mathcal{N},$ $r_{ii}$ is symmetric positive definite and strictly diagonally dominant. The objective function in \eqref{eq13} can be treated as a cost that motivates the sensors to keep connectivity with their neighbors. Different from \cite{Stankovic}, in this paper, we consider the following global objective:
\begin{equation}\label{eq14}
l^g_i(\mathbf{x})=\sum_{j\in\mathcal{N}^p_i}c_{ij}||x_i-x_j||^2,
\end{equation}
where $\mathcal{N}^p_i$ denotes the \textbf{physical neighbor set} of sensor $i$. This basically means that if sensor $j\in \mathcal{N}^p_i$, then, the objective function of sensor $i$ depends on the position of sensor $j$ but sensor $j$ is not necessarily a neighbor of sensor $i$ in the \textbf{communication graph}. The modification is reasoned as follows. If sensor $j$, where $j\in \mathcal{N}^p_i$, is a neighbor of sensor $i$ in the communication graph, then, the corresponding term $||x_i-x_j||^2$ denotes the willingness of sensor $i$  to keep its connectivity with sensor $j$. Else if sensor $j$, where $j\in\mathcal{N}^p_i$, is not a neighbor of sensor $i$ in the communication graph, the corresponding term $||x_i-x_j||^2$ denotes sensor $i$'s willingness to get closer to sensor $j$ (such that it may have a new connection to sensor $j$). By such a modification, the objective function in \eqref{eq13} can be treated as a special case of \eqref{eq14} by enforcing $\mathcal{N}^p_i=\mathcal{N}_i.$

In the following, we consider a network of sensors in which sensor $i$'s objective function is
\begin{equation}\label{eq16}
f_i(\mathbf{x})=l_i^c(x_i)+l_i^g(\mathbf{x}).
\end{equation}

Suppose that the sensors' dynamics are governed by
\begin{equation}\label{fir_dis}
\dot{x}_i=u_i+d_i(t),
\end{equation}
and the control strategies are given in \eqref{eq8}-\eqref{eq9} (by directly adapting $u_i,\hat{z}_i,\hat{x}_i$ therein to two-dimensional column vectors). Then, the following corollary holds.
\begin{Corollary}\label{co2}
Suppose that Assumptions \ref{ass_3}-\ref{ass_4} are satisfied. Then, for any positive constant $v_1$, there exist positive constants $\theta^*,\sigma^*(v_1)$ such that for each $\theta\in(\theta^*,\infty)$, $\sigma\in(\sigma^*,\infty)$, there exists a constant $\bar{T}\geq t_0$ such that
\begin{equation}
||\mathcal{E}(t)||\leq v_1, \ \ \forall t>\bar{T}.
\end{equation}
\end{Corollary}
\begin{proof}
See Section \ref{prof_co2} for the proof.
\end{proof}

Corollary \ref{co2} considers the sensor connectivity game by supposing that the sensors' dynamics are subject to time-varying disturbances only. In the following, we consider the case in which the sensors' dynamics are given by
\begin{equation}\label{fis_un}
\dot{x}_i=u_i+g_i(\mathbf{x})+d_i(t),
\end{equation}
and the control strategies are given in \eqref{eq11}-\eqref{eq12} (by directly adapting $u_i,\hat{z}_i,\hat{x}_i$ therein to two-dimensional column vectors). Then, the following result can be obtained.
\begin{Corollary}\label{co3}
Suppose that Assumptions \ref{ass_3} and \ref{ass_5} are satisfied. Then, for any positive constant $v_1$, there exists a $\theta^*(v_1)>0$ such that for each $\theta\in(\theta^*,\infty)$, there exists a positive constant $k_s^r(v_1,\theta)$ such that for each $k_i^s\in(k_s^r,\infty),i\in\mathcal{N}$, there exists a positive constant $c^r(v_1,\theta,k_s)$ such that for each $c_i\in(c^r,\infty),i\in\mathcal{N},$ there exists a positive constant $\beta^r(v_1,\theta,c,k_s)$ such that for each $\beta_i\in(\beta^r,\infty),i\in\mathcal{N}$
\begin{equation}
||\mathcal{E}_f(t)||\rightarrow 0 \ \ \text{as} \ \ t\rightarrow \infty
\end{equation}
given that $||\mathcal{E}_f(t_0)||\leq v_1.$
\end{Corollary}
 \begin{proof}
 The proof is similar to the proof of Theorem \ref{th2} by replacing $V_4$ therein with
 \begin{equation}
 V_4=\mathbf{\xi}^T\Gamma_1\mathbf{\xi},
 \end{equation}
 where $\Gamma_1$ is defined in the proof of Corollary \ref{co2}.
 \end{proof}

\section{Numerical Examples}\label{p_1_numer}
In this section, we consider a network of $5$ players equipped with the communication graph given in Fig. \ref{d_d1}. In the following, we will simulate the connectivity control game and a non-quadratic game, successively.

\begin{figure}[htb!]
\begin{center}
\vspace{4mm}
\scalebox{0.4}{\includegraphics{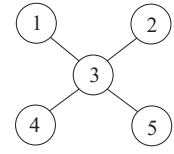}}
\caption{The communication graph among the players.}\label{d_d1}
\end{center}
\end{figure}
\subsection{Connectivity control of sensor networks}
\begin{example}
Consider the connectivity control game in \eqref{eq16} in which $r_{ii}=\left[
                                                                  \begin{array}{cc}
                                                                    i & 0 \\
                                                                    0 & i \\
                                                                  \end{array}
                                                                \right]$, $r_i=[i\ \ i]^T,$ $b_i=i$.
In addition, $l^g_1(\mathbf{x})=||x_1-x_2||^2$, $l^g_2(\mathbf{x})=||x_2-x_3||^2$, $l^g_3(\mathbf{x})=||x_3-x_2||^2$ $l^g_4(\mathbf{x})=||x_4-x_2||^2+||x_4-x_5||^2$ and $l^g_5(\mathbf{x})=||x_5-x_1||^2.$
The game has a unique Nash equilibrium on which $x_{ij}^*=-\frac{1}{2}$ for $i\in\{1,2,\cdots,5\},j\in\{1,2\}.$
\end{example}

In the following, we will firstly simulate the case in which the sensors' dynamics are subject to disturbances only, followed by the case in which the sensors' dynamics are subject to both unmodeled terms and time-varying disturbances.

\subsubsection{Sensors subject to external disturbances in their dynamics}\label{s1}
Suppose that the sensors' dynamics are given by \eqref{fir_dis}
and the external disturbance of sensor $i$ is a sinusoidal function that is of amplitude $i$ and frequency $i$.

Initialized at $\mathbf{x}(0)=[-10,2,-3,-8,-5,6,0,-8,-1,10]^T,$ the simulation results are given in Figs. \ref{trajec_bound}-\ref{obser_dist_bound} by utilizing the method in \eqref{eq8}-\eqref{eq9}. From Fig. \ref{trajec_bound}, it can be concluded that the sensors' positions converge to a small neighborhood of the Nash equilibrium. The disturbance observation errors are given in Fig.  \ref{obser_dist_bound} from which we see that the errors converge to a small neighborhood of zero. Hence, the effectiveness of the proposed method in \eqref{eq8}-\eqref{eq9} is verified.

\begin{figure}[htb!]
\centering
\scalebox{0.5}{\includegraphics{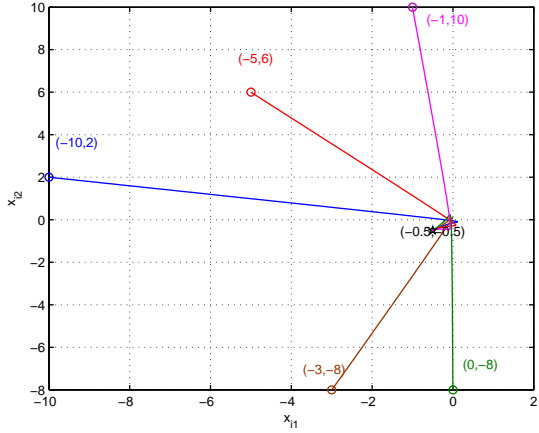}}
\caption{The trajectories of the sensors' positions produced by \eqref{eq8}-\eqref{eq9}.}\label{trajec_bound}
\end{figure}
\begin{figure}[htb!]
\centering
\scalebox{0.5}{\includegraphics{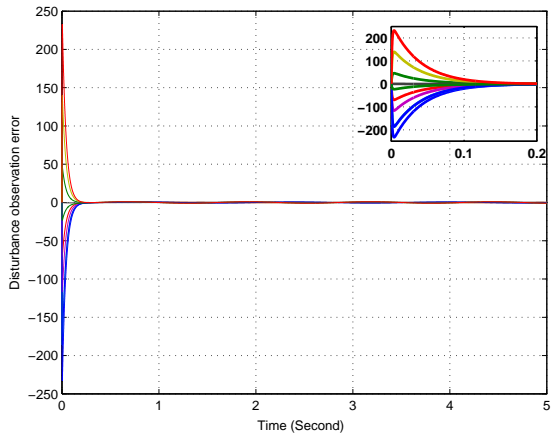}}
\caption{The disturbance observation errors.}\label{obser_dist_bound}
\end{figure}

To further illustrate the robustness of the PI-based method towards disruption of communication among the sensors, we suppose that for $t\in(0.01,2),$ $a_{ij}=0,$ for all $i,j\in\mathcal{N}$ by which the sensors can not receive anything from others. In this case, the trajectories of the sensors' positions  generated by \eqref{eq8}-\eqref{eq9} are plotted in Fig. \ref{case1_lost}, from which we see that the players' actions can still be steered to a small neighborhood of the Nash equilibrium thus numerically verifying the robustness of the PI-based method towards the disruption of communication.
\begin{figure}[htb!]
\centering
\scalebox{0.5}{\includegraphics{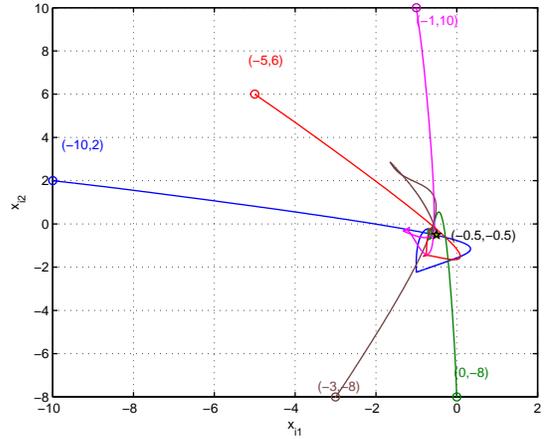}}
\caption{The trajectories of the sensors' positions generated by \eqref{eq8}-\eqref{eq9} under disrupted communication.}\label{case1_lost}
\end{figure}

\subsubsection{Sensors subject to both unmodeled and disturbance terms in their dynamics}\label{s2}

In this section, we suppose that the sensors' dynamics are given by \eqref{fis_un}.

In the dynamics of $x_{ij}$ for $i\in\{1,2,\cdots,5\},j\in\{1,2\}$, the unmodeled and disturbance terms (i.e., $g_i(\mathbf{x})+d_i(t)$) are $\sin(t)+x_{21},\sin(t)+x_{22},2\sin(2t)+x_{11}^2+x_{31},2\sin(2t)+x_{22},3\sin(3t)+x_{31},3\sin(3t)+x_{32},4\sin(4t)+x_{41},4\sin(4t)+x_{42},5\sin(5t)+x_{51},$ and $5\sin(5t)+x_{52},$ respectively. With $\mathbf{x}(0)=[-10,2,-3,-8,-5,6,0,-8,-1,10]^T,$ the simulation results are presented in Figs. \ref{trajec_unmodel}-\ref{obser_dist_unmodel} by utilizing the method in \eqref{eq11}-\eqref{eq12}. Fig. \ref{trajec_unmodel} shows the evolution of the sensors' positions. Fig. \ref{obser_dist_unmodel} depicts the observation errors of the unmodeled and disturbance terms. The simulation results demonstrate that driven by the proposed method, the sensors' positions would converge to the Nash equilibrium.

\begin{figure}[htb!]
\centering
\scalebox{0.5}{\includegraphics{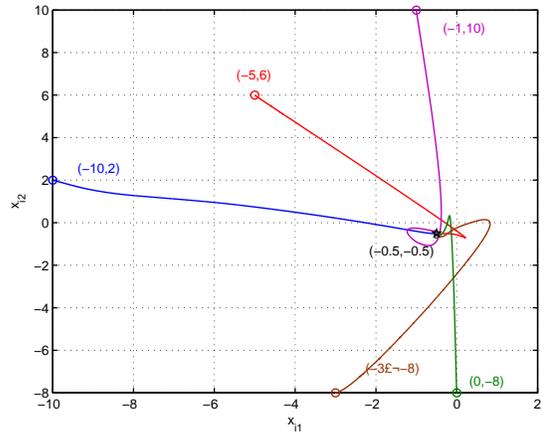}}
\caption{The trajectories of the sensors' positions produced by \eqref{eq11}-\eqref{eq12}.}\label{trajec_unmodel}
\end{figure}
\begin{figure}[htb!]
\centering
\scalebox{0.5}{\includegraphics{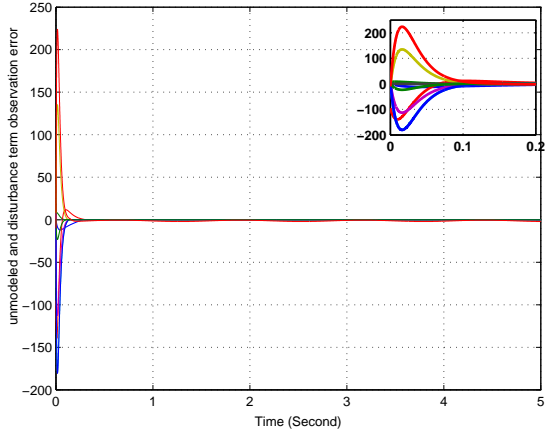}}
\caption{The unmodeled dynamics and disturbance term observation errors.}\label{obser_dist_unmodel}
\end{figure}
Moreover, if the communication among the sensors is lost for $t\in(0.01,2),$ the sensors' positions generated by \eqref{eq11}-\eqref{eq12} are plotted in Fig. \ref{case2_lost}. From Fig. \ref{case2_lost}, it is clear that the sensors' positions can be driven to the Nash equilibrium by the RISE-based method, which numerically verifies the robustness of the RISE-based method towards disrupted communication.
\begin{figure}[htb!]
\centering
\scalebox{0.5}{\includegraphics{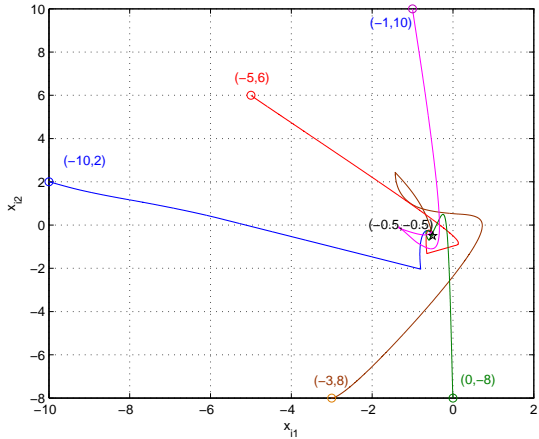}}
\caption{The trajectories of the sensors' positions generated by \eqref{eq11}-\eqref{eq12} under disrupted communication.}\label{case2_lost}
\end{figure}

\subsection{Games with non-quadratic cost functions}
\begin{example}
The example settings are the same as those in Example 1 except that
\begin{equation}
f_1(\mathbf{x})=x_1^Tx_1+x_{11}+x_{12}+1+10e^{x_{11}}+||x_1-x_2||^2.
\end{equation}
Through direct calculation, $\mathbf{x}^*=[-1.2304,-0.5,-0.5,-0.5,-0.5,-0.5,-0.5203,-0.5,$ $-0.6217,-0.5]^T.$
\end{example}

\subsubsection{Games with disturbances in the players' dynamics}
The simulation settings are the same as those in Section  \ref{s1}. The simulation results produced by the method in \eqref{eq8}-\eqref{eq9} are given in Figs. \ref{tra_nonlinear2}-\ref{obser22}. Fig. \ref{tra_nonlinear2} illustrates the players' actions from which we see that the players' actions are driven to a small neighborhood of the Nash equilibrium. Moreover, the observation errors tend to a small neighborhood of zero as shown in Fig. \ref{obser22}. The simulation results verify the theoretical result in Theorem \ref{th1}.
\begin{figure}[htb!]
\centering
\scalebox{0.6}{\includegraphics{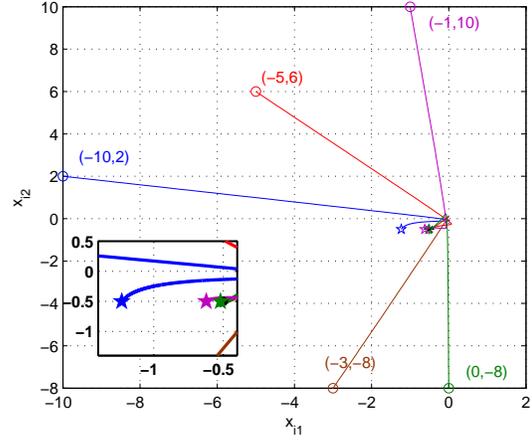}}
\caption{The trajectories of the players' actions produced by \eqref{eq8}-\eqref{eq9}.}\label{tra_nonlinear2}
\end{figure}
\begin{figure}[htb!]
\centering
\scalebox{0.5}{\includegraphics{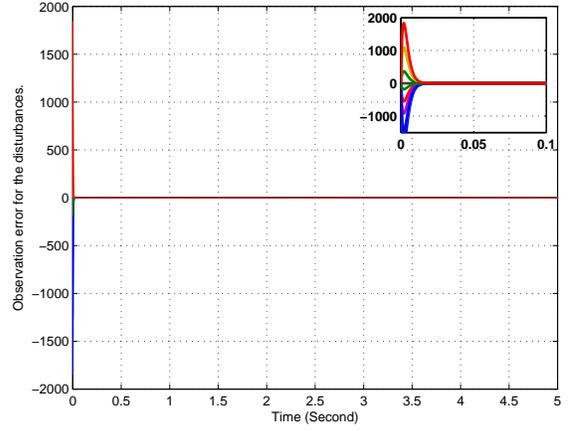}}
\caption{The disturbance observation errors.}\label{obser22}
\end{figure}

Moreover, even if the communication is lost for $t\in(0.01,2),$ the sensors' positions generated by the PI-based method still converge to a small neighborhood of the Nash equilibrium as shown in Fig. \ref{case4_lost}.
\begin{figure}[htb!]
\centering
\scalebox{0.5}{\includegraphics{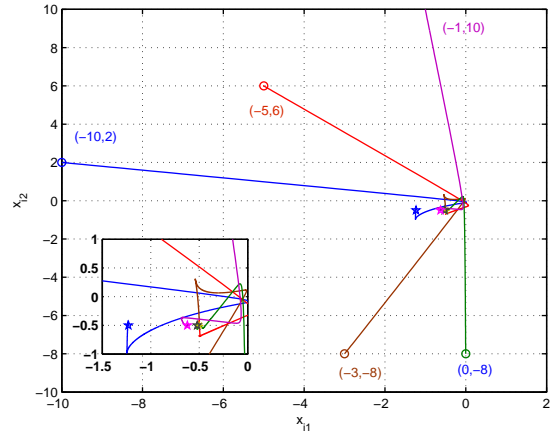}}
\caption{The trajectories of the players' positions generated by \eqref{eq8}-\eqref{eq9} under disrupted communication.}\label{case4_lost}
\end{figure}

\subsubsection{Games with unmodeled and disturbance terms in the players' dynamics}

In this section, the simulation settings follow those in Section \ref{s2}.  The simulation results produced by \eqref{eq11}-\eqref{eq12} are presented in Figs. \ref{tra_nonlinear1}-\ref{obser12}, which plot the players' trajectories and their unmodeled and disturbance term observation errors, respectively. From the simulation results, we see that the observation errors go to zero as $t\rightarrow \infty$ and the players' actions tend to the Nash equilibrium as $t\rightarrow \infty$ thus validating the method in \eqref{eq11}-\eqref{eq12}.

\begin{figure}[htb!]
\centering
\scalebox{0.6}{\includegraphics{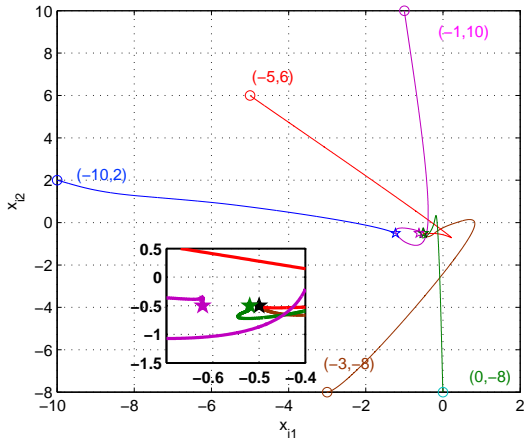}}
\caption{The trajectories of the players' actions produced by \eqref{eq11}-\eqref{eq12}.}\label{tra_nonlinear1}
\end{figure}
\begin{figure}[htb!]
\centering
\scalebox{0.5}{\includegraphics{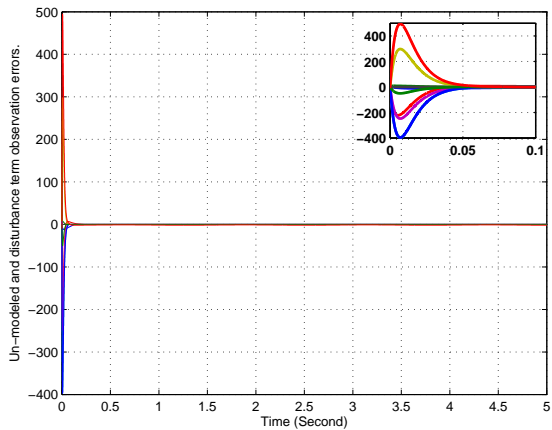}}
\caption{The unmodeled dynamics and disturbance term observation errors.}\label{obser12}
\end{figure}

Furthermore, if the communication is lost for $t\in(0.01,2),$ the sensors' positions generated by the RISE-based method  converge to the Nash equilibrium as shown in Fig. \ref{case3_lost}.
\begin{figure}[htb!]
\centering
\scalebox{0.48}{\includegraphics{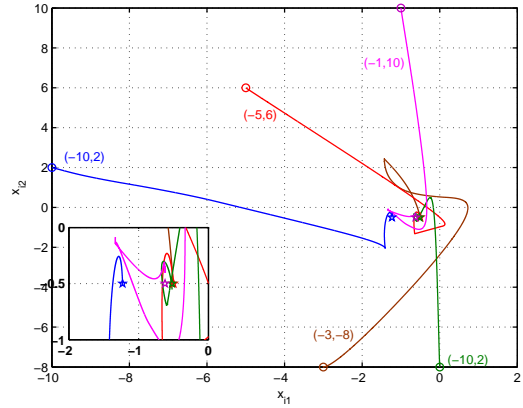}}
\caption{The trajectories of the players' positions generated by \eqref{eq11}-\eqref{eq12} under disrupted communication.}\label{case3_lost}
\end{figure}

\section{Conclusions}\label{conc}
This paper considers robust Nash equilibrium seeking for games in which there are unmodeled and disturbance terms in the players' dynamics. To accommodate the unmodeled and disturbance terms, two extended state observers are introduced into the distributed Nash equilibrium seeking strategy. The first observer is based on PI control. By utilizing the PI-based observer, it is shown that the proposed method can drive the players' actions to a small neighborhood of the Nash equilibrium. Moreover, to further enhance the convergence result, a RISE-based observer is employed in the distributed Nash equilibrium seeking strategy. By Lyapunov stability analysis, it is proven that the method can drive the players' actions to the Nash equilibrium asymptotically under certain conditions. Lastly, the proposed method is applied to the connectivity control of mobile sensor networks. Experimental verifications of the proposed methods and applications in the smart grid (see, e.g., \cite{Hua2019}\cite{Shi2019}) will be considered as future works.

%
%
%
%

\subsection{Proof of Theorem \ref{th1}}\label{prof_th1}
 Define $\bar{\mathbf{\zeta}}_1=\mathbf{\zeta}_1,\bar{\mathbf{\zeta}}_2=\frac{\mathbf{\zeta}_2}{\sigma}.$ Then, it follows from \eqref{eq6} that
\begin{equation}\label{eq7}
\dot{\bar{\mathbf{\zeta}}}=-\sigma \bar{\Delta}\bar{\mathbf{\zeta}}+\frac{\mathbf{d}_1(t)}{\sigma},
\end{equation}
where $\bar{\mathbf{\zeta}}=[\bar{\mathbf{\zeta}}_1^T,\bar{\mathbf{\zeta}}_2^T]^T$ and $\bar{\Delta}=\left[
                    \begin{array}{cc}
                      \mathbf{k}_1 & -I_{N\times N} \\
                      \mathbf{k}_2 & \mathbf{0}_{N\times N} \\
                    \end{array}
                  \right].
$
Then, it can be easily verified that $-\bar{\Delta}$ is Hurwitz. Hence, there exist symmetric positive definite matrices $P_1,Q_1$ such that
\begin{equation}
P_1\bar{\Delta}+\bar{\Delta}^T P_1=Q_1,
\end{equation}
by Theorem 4.6 in \cite{Khalil02}.

According to \eqref{eq7}, the evolution of $\bar{\mathbf{\zeta}}$ is independent of $\mathbf{\eta}$ and $\mathbf{\xi}.$ Hence, define
\begin{equation}
V_1=\bar{\mathbf{\zeta}}^TP_1\bar{\mathbf{\zeta}}.
\end{equation}

Then, for $l_1=2||P_1||$, we can get that
\begin{equation}
\begin{aligned}
\dot{V}_1=&-\sigma \bar{\mathbf{\zeta}}^T Q_1 \bar{\mathbf{\zeta}}+\frac{2\bar{\mathbf{\zeta}}^TP_1 \mathbf{d}_1(t)}{\sigma}\\
\leq & -\sigma \lambda_{min}(Q_1)||\bar{\mathbf{\zeta}}||^2+\frac{l_1}{\sigma}||\bar{\mathbf{\zeta}}||||\mathbf{d}_1(t)||.
\end{aligned}
\end{equation}

By further noticing that
\begin{equation}
\lambda_{min}(P_1)||\bar{\mathbf{\zeta}}||^2\leq V_1 \leq \lambda_{max}(P_1)||\bar{\mathbf{\zeta}}||^2,
\end{equation}
it follows from Theorem 4.19 in \cite{Khalil02} that
\begin{equation}
\begin{aligned}
||\bar{\mathbf{\zeta}}(t)||\leq &\bar{\beta}(||\bar{\mathbf{\zeta}}(t_0)||,\sigma (t-t_0))\\
&+\frac{2l_1 sup_{t\geq t_0}||\mathbf{d}_1(t)||}{\sigma^2 \lambda_{min}(Q_1)} \sqrt{\frac{\lambda_{max}(P_1)}{\lambda_{min}(P_1)}},
\end{aligned}
\end{equation}
where $\bar{\beta}\in \mathcal{KL}.$

Hence, for any $l_2>\frac{2l_1 sup_{t\geq t_0}||\mathbf{d}_1(t)||}{ \lambda_{min}(Q_1)} \sqrt{\frac{\lambda_{max}(P_1)}{\lambda_{min}(P_1)}},$ there exists a positive constant $T\geq t_0$ such that
\begin{equation}
||\bar{\mathbf{\zeta}}(t)||\leq \frac{l_2}{\sigma^2},\forall t>T.
\end{equation}
Recalling the definition of $\bar{\mathbf{\zeta}},$ we can obtain that
\begin{equation}
||\mathbf{\zeta}_2(t)||\leq \frac{l_2}{\sigma},\forall t>T.
\end{equation}

To analyze the evolution of $\mathbf{\eta}$ and $\mathbf{\xi}$, define
\begin{equation}
V_2=\mathbf{\eta}^T\mathcal{P}\mathbf{\eta}+\frac{1}{2}\mathbf{\xi}^T\mathbf{\xi},
\end{equation}
and $\mathcal{P}$ is a symmetric positive definite matrix that satisfies
\begin{equation}
\mathcal{P}(\bar{\Theta}(\mathcal{L}\otimes I_{N\times N}+\mathcal{A}_0))+(\bar{\Theta}(\mathcal{L}\otimes I_{N\times N}+\mathcal{A}_0))^T\mathcal{P}=\mathcal{Q},
\end{equation}
where $\mathcal{Q}$ is a symmetric positive definite matrix \cite{YeTcyber18}.
Then, for $\bar{\phi}$, where $\bar{\phi}=[\mathbf{\eta}^T,\mathbf{\xi}^T]^T$, that belongs to a sufficiently large compact set $D$ that contains the origin, there are positive constants $l_i,\forall i\in\{3,4,\cdots,6\}$ (depending on  $D$) such that,
\begin{equation}
\begin{aligned}
\dot{V}_2\leq &-\theta\lambda_{min}(\mathcal{Q})||\mathbf{\eta}||^2-2\mathbf{\eta}^T\mathcal{P}\mathbf{1}_N\otimes \dot{\mathbf{\xi}}\\
&-m||\mathbf{\xi}||^2-\mathbf{\xi}^T\left[\frac{\partial f_i}{\partial x_i}(\mathbf{y}_i)-\frac{\partial f_i}{\partial x_i}(\mathbf{x})\right]_{vec}+\mathbf{\xi}^T\mathbf{\zeta}_2\\
\leq &-\theta\lambda_{min}(\mathcal{Q})||\mathbf{\eta}||^2+l_3||\mathbf{\eta}||^2+l_4||\mathbf{\eta}||||\mathbf{\xi}||\\
&+l_5||\mathbf{\eta}||||\mathbf{\zeta}_2||-m||\mathbf{\xi}||^2+l_6||\mathbf{\xi}||||\mathbf{\eta}||+||\mathbf{\xi}||||\mathbf{\zeta}_2||,
\end{aligned}
\end{equation}
by utilizing Assumptions \ref{ASS1}-\ref{Assu2}.

Define $A_1=\left[
              \begin{array}{cc}
                \theta \lambda_{min}(\mathcal{Q})-l_3   & -\frac{l_4+l_6}{2} \\
                -\frac{l_4+l_6}{2} & m \\
              \end{array}
            \right].
$
Then,
\begin{equation}\label{eq10}
\begin{aligned}
\dot{V}_2\leq &
-\lambda_{min}(A_1)||\bar{\phi}||^2+l_5||\mathbf{\eta}||||\mathbf{\zeta}_2||+||\mathbf{\xi}||||\mathbf{\zeta}_2||\\
=& -\lambda_{min}(A_1)||\bar{\phi}||^2+(l_5+1)||||\bar{\phi}||||\mathbf{\zeta}_2||
\end{aligned}
\end{equation}
where $\lambda_{min}(A_1)>0$ given that $\theta>\frac{(l_4+l_6)^2+4ml_3}{4m \lambda_{min}(\mathcal{Q})}$.

Hence, by choosing $\theta>\frac{(l_4+l_6)^2+4ml_3}{4m \lambda_{min}(\mathcal{Q})}$,
\begin{equation}
\dot{V}_2\leq -\frac{\lambda_{min}(A_1)}{2}||\bar{\phi}||^2,\forall ||\bar{\phi}||\geq \frac{2(l_5+1)\sup_{t\geq t_0}||\mathbf{\zeta}_2||}{\lambda_{min}(A_1)}.
\end{equation}

Moreover,
\begin{equation}
\min\{\lambda_{min}(\mathcal{P}),\frac{1}{2}\}||\bar{\phi}||^2\leq V_2\leq \max\{\lambda_{max}(\mathcal{P}),\frac{1}{2}\}||\bar{\phi}||^2.
\end{equation}
Choose $r$ such that $B_r\in D$, where $B_r$ is a origin-centered ball whose radius is equal to $r$, and
\begin{equation}
\frac{2(l_5+1)\sup_{t\geq t_0}||\mathbf{\zeta}_2(t)||}{\lambda_{min}(A_1)}<\frac{\min\{\lambda_{min}(\mathcal{P}),\frac{1}{2}\}}{\max\{\lambda_{max}(\mathcal{P}),\frac{1}{2}\}}r.
\end{equation}

Then, for $||\bar{\phi}(t_0)||\leq \frac{\min\{\lambda_{min}(\mathcal{P}),\frac{1}{2}\}}{\max\{\lambda_{max}(\mathcal{P}),\frac{1}{2}\}}r$, there exists a $T_1\geq t_0$ such that for $t_0<t\leq T_1,$
\begin{equation}
||\bar{\phi}(t)||\leq \sqrt{\frac{ \max\{\lambda_{max}(\mathcal{P}),\frac{1}{2}\}}{\min\{\lambda_{min}(\mathcal{P}),\frac{1}{2}\}}}||\bar{\phi}(t_0)||
\end{equation}
and
\begin{equation}
\begin{aligned}
||\bar{\phi}(t)||\leq &\sqrt{\frac{\max\{\lambda_{max}(\mathcal{P}),\frac{1}{2}\}}{\min\{\lambda_{min}(\mathcal{P}),\frac{1}{2}\}}} \frac{2(l_5+1)\sup_{t\geq 0}||\mathbf{\zeta}_2(t)||}{\lambda_{min}(A_1)},\\
&\forall t\geq T_1,
\end{aligned}
\end{equation}
according to Theorem 4.18 in \cite{Khalil02}.

Hence, if $T\leq T_1,$ we have
\begin{equation}
\begin{aligned}
||\bar{\phi}(T)||\leq & \sqrt{\frac{\max\{\lambda_{max}(\mathcal{P}),\frac{1}{2}\}}{\min\{\lambda_{min}(\mathcal{P}),\frac{1}{2}\}}}||\bar{\phi}(t_0)||\\
\leq &\sqrt{\frac{\min\{\lambda_{min}(\mathcal{P}),\frac{1}{2}\}}{\max\{\lambda_{max}(\mathcal{P}),\frac{1}{2}\}}}r.
\end{aligned}
\end{equation}
Moreover, if $T>T_1$, we have
\begin{equation}
\begin{aligned}
||\bar{\phi}(T)||&\leq \sqrt{\frac{\max\{\lambda_{max}(\mathcal{P}),\frac{1}{2}\}}{\min\{\lambda_{min}(\mathcal{P}),\frac{1}{2}\}}} \frac{2(l_5+1)\sup_{t\geq t_0}||\mathbf{\zeta}_2(t)||}{\lambda_{min}(A_1)}\\
&\leq \sqrt{\frac{\min\{\lambda_{min}(\mathcal{P}),\frac{1}{2}\}}{\max\{\lambda_{max}(\mathcal{P}),\frac{1}{2}\}}}r.
\end{aligned}
\end{equation}

Utilizing Theorem 4.18 in \cite{Khalil02} again, we can conclude that there exists $T_2\geq 0$ such that
\begin{equation}
\begin{aligned}
||\bar{\phi}(t)||\leq &\bar{\beta}_2(||\bar{\phi}(T)||,t-T),\forall T\leq t\leq T+T_2\\
||\bar{\phi}(t)||\leq &\sqrt{\frac{\max\{\lambda_{max}(\mathcal{P}),\frac{1}{2}\}}{\min\{\lambda_{min}(\mathcal{P}),\frac{1}{2}\}}} \frac{2(l_5+1)\sup_{t\geq T}||\mathbf{\zeta}_2(t)||}{\lambda_{min}(A_1)}\\
&\forall t\geq T+T_2,
\end{aligned}
\end{equation}
where $\bar{\beta}_2\in\mathcal{KL}.$

Recalling that for $t>T$, $||\mathbf{\zeta}_2(t)||\leq \frac{l_2}{\sigma},$ we arrive at the conclusion.
\subsection{Proof of Theorem \ref{th2}}\label{prof_th2}

To facilitate the closed-loop system analysis, define
\begin{equation}
V=V_1+V_2+V_3+V_3+V_5
\end{equation}
where
$V_1=\frac{1}{2}\mathbf{\zeta}_1^T\mathbf{\zeta}_1,$
$V_2=\frac{1}{2}\mathbf{\gamma}^T\mathbf{\gamma},$
$V_3=\mathbf{\zeta}_1^T(t_0)\beta sgn(\mathbf{\zeta}_1(t_0))-\mathbf{\zeta}_1^T(t_0)N_c(t_0)-\int_{t_0}^{t}(\mathbf{\gamma}^T(N_c(\tau)-\beta sgn(\mathbf{\zeta}_1)))d\tau,$
$V_4=\frac{1}{2}\mathbf{\xi}^T\mathbf{\xi},$ $V_5=\mathbf{\eta}^T\mathcal{P}\mathbf{\eta}$, $N_c(t)=[\dot{d}_i(t)]_{vec}$  and $\mathcal{P}$ is defined in the proof of Theorem \ref{th1}.

Then,
\begin{equation}
\begin{aligned}
\dot{V}_1=&\mathbf{\zeta}_1^T (-(k_s+c)\mathbf{\zeta}_1+\mathbf{\zeta}_2)\\
= &-\mathbf{\zeta}_1^T(k_s+c)\mathbf{\zeta}_1+\mathbf{\zeta}_1^T\mathbf{\zeta}_2,
\end{aligned}
\end{equation}
and
\begin{equation}
\begin{aligned}
&\dot{V}_2=\mathbf{\gamma}^T\left(-k_s\mathbf{\gamma}-\beta sgn(\mathbf{\zeta}_1)+[\dot{g}_i(\mathbf{x})]_{vec}+[\dot{d}_i(t)]_{vec}\right)\\
&=-\mathbf{\gamma}^Tk_s\mathbf{\gamma}+\mathbf{\gamma}^T(-\beta sgn(\mathbf{\zeta}_1)+[\dot{g}_i(\mathbf{x})]_{vec}+[\dot{d}_i(t)]_{vec}).
\end{aligned}
\end{equation}
Moreover, for $\mathcal{E}_f$ that belongs to a compact set $D$ that contains the origin,
\begin{equation}
\dot{V}_3=-\mathbf{\gamma}^T(N_c(t)-\beta sgn(\mathbf{\zeta}_1)),
\end{equation}
and
\begin{equation}
\begin{aligned}
\dot{V}_4=&\mathbf{\xi}^T\left(\left[-\frac{\partial f_i}{\partial x_i}(\mathbf{y}_i)\right]_{vec}+\mathbf{\zeta}_2\right)\\
\leq &-m||\mathbf{\xi}||^2+l_1||\mathbf{\xi}||||\mathbf{\eta}||+||\mathbf{\xi}||||\mathbf{\zeta}_2||,
\end{aligned}
\end{equation}
for some positive constant $l_1$ (depending on $D$) based on Assumption \ref{ASS1}.

Furthermore, by utilizing Assumption \ref{ASS1}, it can be concluded that there are positive constants $l_2,l_3$ and $l_4$ (depending on $D$) such that
\begin{equation}
\begin{aligned}
\dot{V}_5=&-\theta \mathbf{\eta}^T\mathcal{Q}\mathbf{\eta}-2\mathbf{\eta}^T\mathcal{P}\mathbf{1}_N\otimes \dot{\mathbf{x}}\\
\leq &-\theta \lambda_{min}(\mathcal{Q})||\mathbf{\eta}||^2-2\mathbf{\eta}^T\mathcal{P}\mathbf{1}_N\otimes \dot{\mathbf{x}}\\
\leq & -\theta \lambda_{min}(\mathcal{Q})||\mathbf{\eta}||^2+l_2||\mathbf{\eta}||^2\\
&+l_3||\mathbf{\eta}||\mathbf{\xi}||+l_4||\mathbf{\eta}||||\mathbf{\zeta}_2||.
\end{aligned}
\end{equation}

Therefore, for $\mathcal{E}_f\in D$,
\begin{equation}
\begin{aligned}
\dot{V}\leq & -\mathbf{\zeta}_1^T(c+k_s)\mathbf{\zeta}_1+\mathbf{\zeta}_1^T\mathbf{\zeta}_2-\lambda_{min}(k_s)||\mathbf{\gamma}||^2\\
&-m||\mathbf{\xi}||^2+l_1||\mathbf{\xi}||||\mathbf{\eta}||-\theta \lambda_{min}(\mathcal{Q})||\mathbf{\eta}||^2\\
&+l_2||\mathbf{\eta}||^2+l_3||\mathbf{\eta}||\mathbf{\xi}||+l_4||\mathbf{\eta}||||\mathbf{\zeta}_2||+l_5||\mathbf{\gamma}||||\mathbf{\eta}||\\
&+l_5||\mathbf{\gamma}||||\mathbf{\xi}||+l_6||\mathbf{\gamma}||||\mathbf{\zeta}_2||+||\mathbf{\xi}||||\mathbf{\zeta}_2||,
\end{aligned}
\end{equation}
by noticing that for $\mathcal{E}_f\in D$, $\frac{\partial g_i(\mathbf{x})}{\partial x_j}$ for all $i,j\in\mathcal{N}$ are bounded according to Assumption \ref{ass_5}, and there are positive constants $l_5,l_6$ (depending on $D$) such that
\begin{equation}
\begin{aligned}
&\mathbf{\gamma}^T[\dot{g}_i(\mathbf{x})]_{vec}=\mathbf{\gamma}^T\left[\frac{\partial g_i(\mathbf{x})}{\partial \mathbf{x}}\right]_{vec}^T\dot{\mathbf{x}}\\
&\leq l_5||\mathbf{\gamma}||||\mathbf{\eta}||+l_5||\mathbf{\gamma}||||\mathbf{\xi}||+l_6||\mathbf{\gamma}||||\mathbf{\zeta}_2||.
\end{aligned}
\end{equation}

Noticing that $\mathbf{\gamma}=\dot{\mathbf{\zeta}}_1+c \mathbf{\zeta}_1,$ we have,
\begin{equation}
\mathbf{\gamma}=\dot{\mathbf{\zeta}}_1+c \mathbf{\zeta}_1=-k_s\mathbf{\zeta}_1+\mathbf{\zeta}_2.
\end{equation}
Therefore,
\begin{equation}\label{eq1}
\mathbf{\zeta}_2=\mathbf{\gamma}+k_s \mathbf{\zeta}_1.
\end{equation}

Define $A_1=\left[
              \begin{array}{cc}
                m & -\frac{l_1+l_3}{2} \\
                -\frac{l_1+l_3}{2} & \theta \lambda_{min}(\mathcal{Q})-l_2 \\
              \end{array}
            \right]
$ and choose $\theta>\frac{(l_1+l_3)^2+4ml_2}{4m\lambda_{min}(\mathcal{Q})},$ then,

\begin{equation}
\begin{aligned}
\dot{V}\leq &-\lambda_{min}(c)\mathbf{\zeta}_1^T\mathbf{\zeta}_1-(\lambda_{min}(k_s)-l_6)||\mathbf{\gamma}||^2\\
&-\lambda_{min}(A_1)||\mathbf{\phi}||^2+(l_4+1)||\mathbf{\phi}||||\mathbf{\gamma}||\\
&+||\mathbf{\zeta}_1||||\mathbf{\gamma}||+\max\{k_i^s\}(l_4+1)||\mathbf{\phi}||||\mathbf{\zeta}_1||\\
&+l_5||\mathbf{\gamma}||||\mathbf{\eta}||+l_5||\mathbf{\gamma}||||\mathbf{\xi}||+l_6||\mathbf{\gamma}||||k_s\mathbf{\zeta}_1||
\end{aligned}
\end{equation}
where $\mathbf{\phi}=[\mathbf{\xi}^T,\mathbf{\eta}^T]^T.$

Therefore,
\begin{equation}
\begin{aligned}
&\dot{V}\leq  -\left(\lambda_{min}(c)-\frac{1}{2}-\frac{(\max\{k_i^s\})^2(l_4+1)}{2\epsilon_1}\right.\\
&\ \ \ \ \ \ \ \  \ \ \ \ \ \ \ \ \ \  \ \ \ \ \ \ \ \ \ \ \  \ \ \left.-\frac{l_6(\max\{k_i^s\})^2}{2\epsilon_3}\right)||\mathbf{\zeta}_1||^2\\
&-\left(\lambda_{min}(k_s)-l_6-\frac{1}{2}-\frac{l_4+1+2l_5}{2\epsilon_2}-\frac{\epsilon_3l_6}{2}\right)||\mathbf{\gamma}||^2\\
&-\left(\lambda_{min}(A_1)-\frac{(l_4+1+2l_5)\epsilon_2}{2}-\frac{\epsilon_1(l_4+1)}{2}\right)||\mathbf{\phi}||^2,
\end{aligned}
\end{equation}
where $\epsilon_1,\epsilon_2,\epsilon_3$ are positive constants that can be arbitrarily chosen.

Hence, choose $\epsilon_1,\epsilon_2$ to be sufficiently small such that $\lambda_{min}(A_1)-\frac{(l_4+1+2l_5)\epsilon_1}{2}-\frac{\epsilon_2(l_4+1)}{2}>0$ and for fixed $\epsilon_1,\epsilon_2,\epsilon_3$ choose $k_i^s$ to be sufficiently large such that $\lambda_{min}(k_s)-\frac{1}{2}-\frac{l_4+1+2l_5}{2\epsilon_2}-\frac{\epsilon_3l_6}{2}>0$. Moreover, for fixed $\epsilon_1,\epsilon_2,\epsilon_3,k_s,$ choose $c_i$ to be sufficiently large such that $\lambda_{min}(c)-\frac{1}{2}-\frac{(\max\{k_i^s\})^2(l_4+1)}{2\epsilon_1}-\frac{l_6(\max\{k_i^s\})^2}{2\epsilon_3}>0$.

Then, for $\epsilon_4=\min\{\lambda_{min}(A_1)-\frac{(l_4+1+2l_5)\epsilon_1}{2}-\frac{\epsilon_2(l_4+1)}{2},\lambda_{min}(k_s)-l_6-\frac{1}{2}-\frac{l_4+1+2l_5}{2\epsilon_2}-\frac{\epsilon_3l_6}{2},\lambda_{min}(c)-\frac{1}{2}-\frac{(\max\{k_i^s\})^2(l_4+1)}{2\epsilon_1}-\frac{l_6(\max\{k_i^s\})^2}{2\epsilon_3}\},$ we have
\begin{equation}\label{eq}
\dot{V}\leq -\epsilon_4 ||\mathbf{\psi}||^2,
\end{equation}
where $\mathbf{\psi}=[\mathbf{\phi}^T,\mathbf{\zeta}_1^T,\mathbf{\gamma}^T]^T.$

Moreover, for fixed $c$, choose $\beta_i\geq\frac{\max\{c_i\}\sup_{t\geq t_0}||[\dot{d}_i(t)]_{vec}||_1}{\min\{c_i\}}+\frac{\sup_{t\geq t_0}||[\ddot{d}_i(t)]_{vec}||_1}{\min\{c_i\}},i\in \mathcal{N}$ then $V_3\geq 0$ by following the proof of Lemma 5 in  \cite{HuSCL12}. If this is the case, recalling the definition of $V$, we have $\min\{\frac{1}{2},\lambda_{min}(\mathcal{P})\}||\mathcal{E}_p(t)||^2\leq V\leq \max\{\frac{1}{2},\lambda_{max}(\mathcal{P})\}||\mathcal{E}_p(t)||^2,$
where $\mathcal{E}_p(t)=[\mathbf{\zeta}_1^T,\mathbf{\gamma}^T,\sqrt{V_3},\mathbf{\xi}^T,\mathbf{\eta}^T]^T.$ Furthermore, $\dot{V}$ is negative semi-definite and from \eqref{eq}, $\mathbf{\psi}$ is bounded for $t\in[t_0,\infty).$  Therefore, according to \eqref{eq1}, $\mathbf{\zeta}_2$ is bounded, from which we can further conclude that $\dot{\mathbf{x}}$ is bounded according to \eqref{eq5}. Moreover, $\dot{\mathbf{\zeta}}_1$ and $\dot{\mathbf{\zeta}}_2$ are bounded according to \eqref{eq2}. In addition, by \eqref{eq3} and \eqref{eq4}, we can derive that $\dot{\mathbf{\gamma}}$ and $\dot{\mathbf{\eta}}$ are also bounded. Lastly, by utilizing \eqref{eq5}, it's easy to verify that $\ddot{\mathbf{x}}$ is also bounded.

Recalling that
\begin{equation}
\begin{aligned}
\dot{V}=&-\mathbf{\zeta}_1^T(k_s+c)\mathbf{\zeta}_1+\mathbf{\zeta}_1^T\mathbf{\zeta}_2-\mathbf{\gamma}^Tk_s\mathbf{\gamma}\\
&+\mathbf{\xi}^T\left(-\left[\frac{\partial f_i}{\partial x_i}(\mathbf{y}_i)\right]_{vec}+\mathbf{\zeta}_1\right)\\
&-\mathbf{\eta}^T\mathcal{Q}\mathbf{\eta}-2\mathbf{\eta}^T\mathcal{P}\mathbf{1}_N\otimes \dot{\mathbf{x}}+\mathbf{\gamma}^T[\dot{g}_i(\mathbf{x})]_{vec}.
\end{aligned}
\end{equation}
It can be easily verified that $\ddot{V}$ is bounded. Hence, by the Babalat's Lemma, $||\mathbf{\psi}||\rightarrow 0$ as $t\rightarrow \infty.$ Recalling that $\mathbf{\zeta}_2=\mathbf{\gamma}+k_s\mathbf{\zeta}_1$, we arrive at the conclusion.

\subsection{Proof of Corollary \ref{co2}}\label{prof_co2}
Define
\begin{equation}\nonumber
\mathcal{R}=\left[
                      \begin{array}{cccc}
                        2r_{11}+h_{11} & C_{12} & \cdots & C_{1N} \\
                        C_{21} & 2r_{22}+h_{22} & \cdots & C_{2N} \\
                        \vdots& \vdots & \ddots & \vdots \\
                        C_{N1} & C_{N2} & \cdots & 2r_{NN}+h_{NN} \\
                      \end{array}
                    \right],
\end{equation}
where $h_{ii}=\left[
                \begin{array}{cc}
                  2\sum_{j\in \mathcal{N}_i^p} c_{ij} & 0 \\
                  0 & 2\sum_{j\in \mathcal{N}_i^p} c_{ij} \\
                \end{array}
              \right]$, $C_{ij}=\left[
                                      \begin{array}{cc}
                                        -2c_{ij} & 0 \\
                                        0 & -2c_{ij} \\
                                      \end{array}
                                    \right]
              $, $c_{ij}=0$ if $j\notin \mathcal{N}_i^p,$ and $c_{ij}>0$ if $j\in \mathcal{N}_i^p.$
Then, it can be easily seen that $\mathcal{R}$ is strictly diagonally dominant with its diagonal elements being positive as $r_{ii}$ is symmetric positive definite and strictly diagonally dominant. Therefore, there are symmetric positive definite matrices $\Gamma_1,\Gamma_2$ that satisfy
\begin{equation}
\Gamma_1 \mathcal{R}+\mathcal{R}^T\Gamma_1=\Gamma_2,
\end{equation}
according to the Gershgorin Circle Theorem and Theorem 4.6 in \cite{Khalil02}.

The rest of the proof follows the proof of Theorem \ref{th1} with $V_2$ therein replaced by
\begin{equation}
V_2=\mathbf{\eta}^T\mathcal{P}\mathbf{\eta}+\mathbf{\xi}^T\Gamma_1 \mathbf{\xi},
\end{equation}
and the details are omitted here.
%


\begin{thebibliography}{99}
\bibitem{YeTAC18}M. Ye, G. Hu, ``Distributed Nash equilibrium seeking by a consensus based approach," \emph{IEEE Transactions on Automatic Control}, vol. 62, no. 9, pp. 4811-4818, 2017.
\bibitem{SalehisadaghianiAT16} F. Salehisadaghiani and L. Pavel, ``Distributed Nash equilibrium seeking:
A gossip-based algorithm," \emph{Automatica}, vol. 72, pp. 209-216, 2016.
\bibitem{YeAT18}M. Ye, G. Hu, and F. L. Lewis, ``Nash equilibrium seeking for N-coalition non-cooperative games," \emph{Automatica}, vol. 95, pp. 266-272, 2018.
\bibitem{YETAC19} M. Ye, G. Hu, F. L. Lewis, L. Xie, ``A unified strategy for solution seeking in graphical N-coalition noncooperative games," \emph{IEEE Transactions on Automatic Control,} vol. 64, no. 11, pp. 4645-4652, 2019.
\bibitem{YEAT2020} M. Ye, G. Hu and S. Xu, ``An extremum seeking-based approach for Nash equilibrium seeking in N-cluster noncooperative games," \emph{Automatica}, vol. 114, 108815, 2020.

\bibitem{YECyber2017} M. Ye, G. Hu, ``Game design and analysis for price based demand response: an aggregate game approach," \emph{IEEE Transactions on Cybernetics,} vol. 47, no. 3, pp. 720-730, 2017.
\bibitem{KoshalOR16}J. Koshal, A. Nedic and U. Shanbhag, ``Distributed algorithms for aggregative games on graphs," \emph{Operations Research,} vol. 64, pp. 680-704, 2016.

\bibitem{Belgioioso17}G. Belgioioso and S. Grammatico, ``Semi-decentralized Nash equilibrium seeking in aggregative games with separable coupling constraints and non-differentiable cost functions," \emph{IEEE Control Systems Letters,} vol. 1, no. 2, pp. 400-405, 2017.
\bibitem{ZhuAUTO16}M. Zhu, and E. Frazzoli, ``Distributed robust adaptive equilibrium computation for generalized convex games," \emph{Automatica}, vol. 63, pp. 82-91, 2016.
\bibitem{YeTcyber18}M. Ye, G. Hu, ``Distributed Nash equilibrium seeking in multi-agent games under switching communication topologies," \emph{IEEE Transactions on Cybernetics}, vol. 48, no. 11, pp. 3208-3217, 2018.
\bibitem{Yecyber16}M. Ye, G. Hu, ``Solving potential games with dynamical constraint," \emph{IEEE Transactions on Cybernetics,} vol. 46, no. 5, pp. 1156-1164, 2016.
\bibitem{ZhengTCST09}Q. Zheng, L. Dong, D. Lee and Z. Gao, ``Active disturbance rejection control for MEMS gyroscopes," \emph{IEEE Transactions on Control Systems Technology,} vol. 17, no. 6, pp. 1432-1438, 2009.
\bibitem{YaoTASCE17}J. Yao, W. Deng and Z. Jiao, ``RISE-based adaptive control of hydraulic systems with asymptotic tracking," \emph{IEEE Transactions on Automation Science and Engineering,} vol. 14, no. 3, pp. 1524-1531, 2017.
\bibitem{Fazeli12} A. Fazeli, M. Zeinali and A. Khajepour, ``Applications of adaptive sliding mode control for regenerative braking torque control," \emph{IEEE/ASME Transactions on Mechatronics,} vol. 17, no. 4, pp. 745-755, 2012.
\bibitem{YuanTIE19}Y. Yuan, Y. Yu and L. Guo, ``Nonlinear active disturbance rejection control for the pneumatic muscle actuators with discrete-time measurements," \emph{IEEE Transactions on Industrial Electronics,} vol. 66, no. 3, pp. 2044-2053, 2019.
\bibitem{HeTIE19} S. He, S. Dai and F. Luo, ``Asymptotic trajectory tracking control with guaranteed transient behavior for MSV with uncertain dynamics and external disturbances," \emph{IEEE Transactions on Industrial Electronics,} vol. 66, pp. 3712-3720, no. 5, 2019.
\bibitem{YeAT16}M. Ye, G. Hu, ``A robust extremum seeking scheme for dynamic systems with uncertainties and disturbances," \emph{Automatica}, vol. 66, pp. 172-178, 2016.
\bibitem{HuSCL12}G. Hu, ``Robust consensus tracking of a class of second-order multi-agent dynamic systems," \emph{Systems and Control Letters,} vo. 61, no. 1, pp. 134-142, 2012.
\bibitem{XianTAC04} B. Xian, D. M. Dawson, M. S. de Queiroz, and J. Chen, ``A continuous
asymptotic tracking control strategy for uncertain nonlinear systems,"
\emph{IEEE Transactions Automatic Control}, vol. 49, no. 7, pp. 1206-1211, 2004.
\bibitem{Khalil02} H. Khailil, \emph{Nonlinear Systems,} Upper Saddle River, NJ: Prentice Hall, 2002.
\bibitem{YETCST16}M. Ye, G. Hu, ``Distributed extremum seeking for constrained networked optimization and its application to energy consumption control in smart grid," \emph{IEEE Transactions on Control Systems Technology}, vol. 24, no. 6, pp. 2048-2058, 2016.

\bibitem{Stankovic}M. Stankovic, K. Johansson and D. Stipanovic, ``Distributed seeking
of Nash equilibria with applications to mobile sensor networks," \emph{IEEE
Transactions on Automatic Control}, vol. 57, no. 4, pp. 904-919, 2012.
\bibitem{Chen04}W. Chen, ``Disturbance observer based control for nonlinear systems," \emph{IEEE/ASME Transactions on Mechatronics}, vol. 9, no. 4, pp. 706-710, 2004.
\bibitem{SheTIE08}J. She, M. Fang, Y. Ohyama, H. Hashimoto, M. Wu, ``Improving disturbance-rejection performance based on an equivalent-input-disturbance approach," \emph{IEEE Transactions on Industrial Electronics}, vol. 55, no. 1, pp. 380-389, 2008.
\bibitem{Chen16}W. Chen, J. Yang, L. Guo and S. Li, ``Disturbance-observer-based control and related methods--an overview," \emph{IEEE Transactions on Industrial Electronics,} vol. 63, no. 2, pp. 1083-1095, 2016.
\bibitem{Li12}S. Li, J. Yang, W. Chen and X. Chen, ``Generalized extended state observer based control for systems with mismatched uncertainties," \emph{IEEE Transactions on Industrial Electronics,} vol. 59, no. 12, pp. 4792-4802, 2012.

\bibitem{YECDC19} M. Ye, ``A RISE-based distributed robust Nash equilibrium seeking
strategy for networked games," \emph{IEEE Conference on Decision and
Control,} 2019, pp. 4047-4052, 2019.

\bibitem{Hua2019} H. Hua, Y. Qin, C. Hao, and J. Cao, ``Optimal energy management
strategies for energy Internet via deep reinforcement learning approach,"
\emph{Applied Energy,} vol. 239, pp. 598-609, 2019.
\bibitem{Shi2019} Y. Shi, H. Tuan, A. Savkin, T. Duong and H. Poor, ``Model predictive
control for smart grids with multiple electric-vehicle charging stations,"
\emph{IEEE Transactions on Smart Grid,} vol. 10, no. 2, pp. 2127-2136,
2019.

\end{thebibliography}
\end{document}